\documentclass[a4paper,12pt,oneside]{article}
\usepackage[latin1]{inputenc}
\usepackage[english]{babel}
\usepackage{amsthm}
\usepackage{amssymb}
\usepackage{amsmath}
\usepackage{hyperref}
\usepackage{mathrsfs}
\usepackage[all]{xy}

\newtheoremstyle{note}{12pt}{12pt}{}{}{\bfseries}{.}{.5em}{}
\title{\LARGE\textbf{A Phase Transition for Circle Maps and Cherry Flows}}
\author{Liviana Palmisano\\ Universit\'e Paris-Sud\\ 91405 Orsay, France}
\newtheorem{theo}[equation]{Theorem}
\newtheorem{prop}{Proposition}

\newtheorem{defin}[equation]{Definition}

\newtheorem{cor}[equation]{Corollary}
\newtheorem{fact}[equation]{Fact}
\numberwithin{equation}{section}
\newtheorem{lemma}[equation]{Lemma}
\def\HD{{\mathrm{dim_H}}}

\begin{document}
\maketitle
\author
\begin{abstract}
We study $C^{2}$ weakly order preserving circle maps with a flat interval. The main result of the paper is about a sharp transition from degenerate geometry to bounded geometry depending on the degree of the singularities at the boundary of the flat interval. We prove that the non-wandering set has zero Hausdorff dimension in the case of degenerate geometry and it has Hausdorff dimension strictly greater than zero in the case of bounded geometry.
Our results about circle maps allow to establish a sharp phase transition in the dynamics of Cherry flows. 
\end{abstract}
\section{Introduction}
\subsection{Almost Smooth Maps with a Flat Interval}\label{class}

We consider the class $\mathscr{L}$ of continuous circle endomorphisms $f$ of degree one for which an open arc $U$ exists such that the following properties hold:\newline

\begin{enumerate}
\item The image of $U$ is one point.\newline
\item The restriction of $f$ to $\mathbb{S}^{1}\setminus{\overline{U}}$ is a $C^{2}$-diffeomorphism onto its image.\newline
\item Let $\left(a,b\right)$ be a preimage of $U$ under the projection of the real line of $\mathbb{S}^{1}$.
On some right-sided neighborhood of $b$, $f$ can be represented as
\begin{displaymath}
h_{r}\left(\left(x-b\right)^{l_{r}}\right),
\end{displaymath}
where $h_{r}$ is a $C^{2}$-diffeomorphism on an open neighborhood of $b$.\\
Analogously, there exists a $C^{2}$-diffeomorphism on an left-sided neighborhood of $a$ such that $f$ is of the form
\begin{displaymath}
h_{l}\left(\left(a-x\right)^{l_{l}}\right).
\end{displaymath}
\end{enumerate}
The ordered pair $\left(l_{l},l_{r}\right)$ will be called the $critical$ $exponent$ of the map. If $l_{l}=l_{r}$ the map will be referred to as $symmetric$.\newline
\subsection{Rotation Number}
Let $f:\mathbb{S}^1\rightarrow\mathbb{S}^1$ be a map in $\mathscr{L}$ and $F:\mathbb{R}\rightarrow\mathbb{R}$ a lift of $f$. Then
\begin{displaymath}
\rho(F):=\lim_{n\to\infty}\frac{1}{n}\left(F^{n}(x)-x\right)
\end{displaymath}
exists for all $x\in\mathbb R$. \\
$\rho(F)$ is independent of $x$ and well defined up to an integer; that is, if $\tilde{F}$ is another lift of $f$, then $\rho(F)-\rho(\tilde{F})=F-\tilde{F}\in\mathbb Z$.\\
This remark justifies the following terminology:
\begin{defin}
$\rho(f):=\left[\rho\left(F\right)\right]$ is called the $rotation$ $number$ of $f$.
\end{defin}

In the discussion that follows in this paper, it will often be convenient to identify $f$ and $F$ and subsets of $\mathbb S^1$ with the corresponding subsets of $\mathbb R$.\\
Also we permanently assume that the rotation number is $irrational$. 
\subsection{Combinatorics}
The irrational rotation number $\rho(f)$ can be written as an infinite continued fraction
\begin{displaymath}
\rho(f)=\frac{1}{a_1+\frac{1}{a_2+\frac{1}{\cdots}}}.
\end{displaymath}
where $a_i$ are positive integers.\\

If we cut off the portion of the continued fraction beyond the $n$-th position, and write the resulting fraction in lowest terms as $\frac{p_n}{q_n}$ then the numbers $q_n$ for $n\geq 1$ satisfy the recurrence relation
\begin{displaymath}
q_{n+1} = a_{n+1}q_n+q_{n-1};\textrm{  } q_0 = 1;\textrm{  } q_1 = a_1.
\end{displaymath}

The number $q_n$ is the iterate of the rotation by $\rho(f)$ for which the orbit of any point makes the closest return so far to the point itself.
\subsection{Discussion of the Results}
Our class $\mathscr{L}$ is designed to comprise the first return maps for smooth Cherry flows (see section \ref{Cherry flow}). Following the terminology for these flows we will distinguish three cases and we will classify the maps in $\mathscr{L}$ depending on their critical exponent $\left(l,l\right)$:
\begin{itemize}
\item the non-dissipative case if $l<1$;
\item the conservative case if $l=1$;
\item the dissipative case if $l>1$.
\end{itemize}

In \cite{e} which deals with the non-dissipative case it is proved that the non-wandering set $\Omega$ has zero Hausdorff dimension. In the conservative case, $\HD(\Omega)=0$, \cite{s}. The dissipative case is studied in \cite{on} and \cite{y}. In particular in \cite{y}, it is proved that $\Omega$ has zero Lebesgue measure. In the dissipative case, under the assumption that the rotation number is of bounded type (i.e. $\sup_i a_i<\infty$) $\HD(\Omega)<1$, \cite {a}.\newline
\subsection{Statement of Results}
We investigate symmetric almost smooth maps with a flat interval and with critical exponent $(l,l)$, $l > 1$ (dissipative case) and we get results about their non-wandering set, which, by Theorem C in \cite{on}, is exactly $\mathbb{S}^{1}\setminus\bigcup^{\infty}_{i=0}f^{-i}(U)$. \\ 
Also, we permanently assume that the rotation number is irrational and of bounded type.
\begin{theo}\label{1}
If $f$ is a map from the class $\mathscr{L}$ with the critical exponent $l\in \left(1,2\right]$, then the Hausdorff dimension of the non-wandering set is equal to 0.
\end{theo}
\begin{theo}\label{2}
If $f$ is a map from the class $\mathscr{L}$ with the critical exponent $l> 2$, then the Hausdorff dimension of the non-wandering set is strictly greater than 0.
\end{theo}
These results show a transition of geometry in the dissipative case as the critical exponent of our map passes through $2$.\newline
As a corollary to Theorem \ref{2} we find an important property of Cherry flows:
\begin{theo}\label{3}
Let $X$ be a Cherry vector field with  $\lambda_1>0>\lambda_2$ as eigenvalues of the saddle point and let $f$ be its first return map.\\
If  $\left|\lambda_2\right|>2\lambda_1$ and if $f$ has rotation number of bounded type, then the quasi-minimal set has Hausdorff dimension strictly greater than $1$. 
\end{theo}
\section{Technical Tools}
\subsection{Basic Notations} 
We will introduce a simplified notation for backward and forward images of the flat interval $U$. Instead of $f^{i}\left(U\right)$ we will simply write $\underline{i}$; for example, $\underline{0}=U$. Thus, for us, underlined positive integers represent points, and underlined non-positive integers represent intervals.
\subsection {Distance between Points}
We denote by $\left(a,b\right)=\left(b,a\right)$ the open shortest interval between $a$ and $b$ regardless of the order of these two points. The length of that interval in the natural metric on the circle will be denoted by $\left|a-b\right|$. Following \cite{a}, let us adopt now these conventions of notation:\newline
\begin{itemize}
\item $\left|\underline{-i}\right|$ stands for the length of the interval $\underline{-i}$.
\item Consider a point $x$ and an interval $\underline{-i}$ not containing it. Then the distance from $x$ to the closer endpoint of $\underline{-i}$ will be denoted by $\left|\left(x,\underline{-i}\right)\right|$, and the distance to the more distant endpoint by $\left|\left(x,\underline{-i}\right]\right|$.
\item We define the distance between the endpoints of two intervals $\underline{-i}$ and $\underline{-j}$ analogously. For example, $\left|\left(\underline{-i},\underline{-j}\right)\right|$ denotes the distance between the closest endpoints of these two intervals while $\left|\left[\underline{-i},\underline{-j}\right)\right|$ stands for $\left|\underline{-i}\right|+\left|\left(\underline{-i},\underline{-j}\right)\right|$.
\end{itemize}
\section {Scaling near Critical Point}
We define a sequence of scalings
\begin{displaymath}
\tau_{n}:=\frac{\left|(\underline 0,\underline{q_n})\right|}{\left|(\underline 0,\underline{q_ {n-2}})\right|}. 
\end{displaymath}
These quantities measure `the geometry' near to a critical point.\newline
When $\tau_n\rightarrow 0$ we say that the geometry of the mapping is `degenerate'.\\
When $\tau_n$ is bounded away from zero we say that the geometry is `bounded'. \\

The following sequence will also be frequently used:
\begin{displaymath}
\alpha_n:=\frac{\left|(\underline{-q_n},\underline{0})\right|}{\left|[\underline{-q_ {n}},\underline{0})\right|}.
\end{displaymath}
Along with $\tau_n$ these ratios serve as scalings relating the geometries of successive dynamic partitions and moreover $\alpha_n>\tau_n$.\\

The proof of our results is based on the following theorem.\\
\begin{theo}\label{sw}
\begin{enumerate}
\item If the critical exponent $l\leq 2$, then the scalings $\tau_n$ tend to zero at least exponentially fast.
\item For maps with rotation number of bounded type and with critical exponent $l > 2$ the sequence $\tau_n$ is bounded away from zero.
\end{enumerate}
\end{theo}
\begin{proof}
The second claim is completely proved in \cite{a}, while the first claim is proved under the additional assumption that the Schwarzian derivative is strictly smaller than zero. For the general case, the details are provided in the Appendix.  
\end{proof}

The next proposition will be used to prove the mains results of this paper.
\begin{prop}\label{Koebe liv}
Let $f$ be a function in $\mathscr{L}$ and let $J\subset T$ be two intervals of the circle.\\
Suppose that, for some $n\in\mathbb N$
\begin{itemize}
\item[-] $f^n$ is a diffeomorphism on $T$;
\item[-] $\sum_{i=0}^{n-1}\left|f^{i}(J)\right|$ is bounded;
\item[-] $\left|f^n\left(J\right)\right|\leq K dist\left(f^n\left(J\right),\partial f^n\left(T\right)\right)$ with $K$ a positive constant.
\end{itemize}
 Then, there exists a constant $C$ such that, for every two intervals $A$ and $B$ in $J$
\begin{displaymath}
\frac{\left|f^n\left(A\right)\right|}{\left|f^n\left(B\right)\right|}\geq C\frac{\left|A\right|}{\left|B\right|}.
\end{displaymath}

\end{prop}
\begin{proof}
It follows from the Koebe principle of \cite{h}.\\
\end{proof}
\subsection {Continued Fractions and Partitions}
Since $f$ is order-preserving and has no periodic points, there exists an order-preserving and continuous map $h:\mathbb{S}^{1}\rightarrow\mathbb{S}^{1}$ such that $h\circ f=R_{\rho}\circ h$, where $\rho$ is the rotation number of $f$ and $R_{\rho}$ is the rigid rotation by $\rho$. \\
In particular, the order of points in an orbit of $f$ is the same as the order of points in an orbit of $R_{\rho}$. Therefore, results about $R_{\rho}$ can be translated into results about $f$, via the semiconjugacy $h$.\newline

We can build the so called dynamical partitions $\mathscr{P}_{n}$ of $\mathbb S^1$ to study the geometric properties of $f$, see \cite{fun}.\newline
$\mathscr{P}_{n}$ is generated by the first $q_{n}+q_{n+1}$ preimages of $U$ and consists of 
\begin{displaymath}
\mathscr{I}_{n}:=\left\{ \underline{-i}: 0\leq i\leq q_{n+1}+q_{n}-1\right\},
\end{displaymath}
together with the gaps between these intervals.\newline
There are two kinds of gaps:
\begin{itemize}
\item The `long' gaps are the interval $I^{n}_{0}$, which is the interval between $\underline{-q_{n}}$ and $\underline 0$ for $n$ even or the interval between $\underline{0}$ and $\underline{-q_n}$ for $n$ odd, with its preimages,
\begin{displaymath}
I_{i}^{n}:=f^{-i}(I^{n}_{0}), i=0,1,\ldots q_{n+1}-1.
\end{displaymath}
\end{itemize}
\begin{itemize}
\item The `short' gaps are the interval $I^{n+1}_{0}$, which is the interval between $\underline 0$ and $\underline{-q_{n+1}}$ for $n$ even or the interval between $\underline{ -q_{n+1}}$ and $\underline{0}$ for $n$ odd, with its preimages,
\begin{displaymath}
I_{i}^{n+1}:=f^{-i}(I^{n+1}_{0}), i=0,1,\ldots, q_{n}-1.
\end{displaymath}
\end{itemize}
We will briefly explain the structure of the partitions. \\
Take two consecutive dynamical partitions of order $n$ and $n+1$. The latter is clearly a refinement of the former. All `short' gaps of $\mathscr P_{n}$ become `long' gaps of $\mathscr P_{n+1}$ while all `long' gaps of $\mathscr P_{n}$ split into $a_{n+2}$ preimages of $U$ and $a_{n+2}$ `long' gaps and one `short' gap of the next partition $\mathscr P_{n+1}$:
\begin{equation}\label{div}
I_{i}^{n}=\bigcup_{j=1}^{a_{n+2}}f^{-i-q_n-jq_{n+1}}(U)\cup\bigcup_{j=0}^{a_{n+2}-1}I_{i+q_n+jq_{n+1}}^{n+1}\cup I_{i}^{n+2}.
\end{equation}

Several of the proofs in the following will depend strongly on the relative positions of the points and intervals of $\mathscr{P}_{n}$. \\
In reading the proofs the reader is advised to keep the following pictures in mind, which show some of these objects near the flat interval $\underline{0}$.\newline
\newline
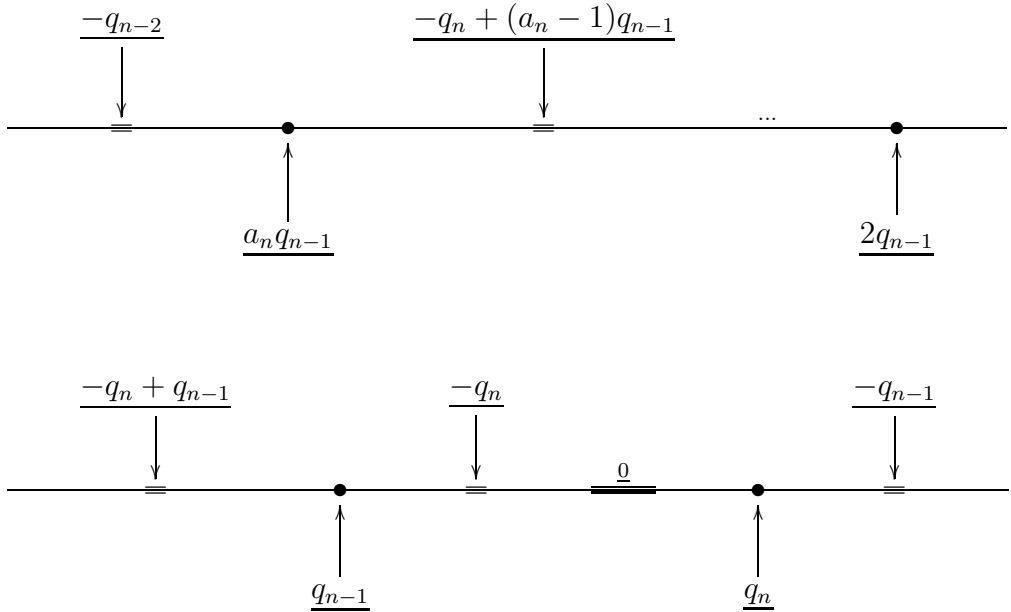
\xymatrix{
&{\underline{-q_{n-2}}} \ar[d] &  &{\underline{-q_{n}+(a_{n}-1)q_{n-1}}} \ar[d]  
 \\
\ar@{-}[rrrrrr]&{=}
&{\bullet}&{=}&{-}\ar@{-}\ar@{-}^{...}&{\bullet}&\\
& & {\underline{a_{n}q_{n-1}}}\ar[u] & & &  {\underline{2q_{n-1}}\ar[u] }}
\vspace{1cm}
\begin{center}
\xymatrix{
&{\underline{-q_n+q_{n-1}}} \ar[d] & &{\underline{-q_{n}}} \ar[d]&&&&{\underline{-q_{n-1}}}\ar[d]
 \\
\ar@{-}[rrrrrrr]&{=}&{\bullet}&{=}&\ar@{=}[r]^{\underline{0}}& &{\bullet}&{=}&\ar@{-}[llll]\\
  & &{\underline{q_{n-1}}}\ar[u] & & & & {\underline{q_{n}}}\ar[u]&  }
\end{center}

\vspace{1cm}

In the next picture we have enlarged the right-hand part of this picture to show the location of the points $\underline{q_{n}}$, $\underline{2q_{n}}$ and $\underline{3q_{n}}$ for the case $n$ even and $a_n=1$.
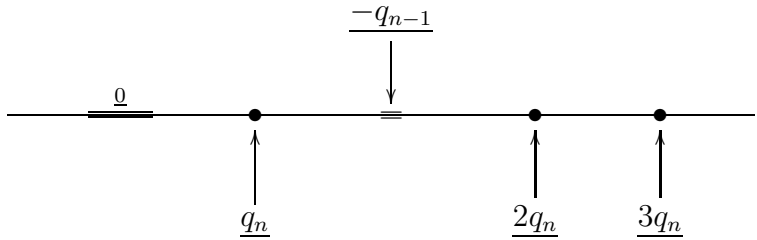
\begin{center}
\xymatrix{
& & &&{\underline{-q_{n-1}}} \ar[d]&
 \\
\ar@{-}[rrrrrrr]&\ar@{=}[r]^{\underline{0}}&&{\bullet}&{=}&{\bullet}&{\bullet}&\\
  && &{\underline{q_{n}}}\ar[u] &  &{\underline{2q_{n}}}\ar[u]& {\underline{3q_{n}}}\ar[u]&  }
\end{center}
Now we present a few results from \cite{a} which are used throughout the paper:
\begin{prop}\label{figo}
Let $A\in\mathscr I_{n}$ be a preimage of $U$ and $B$ any of two gaps adjacent to $A$. Then there exists a constant $C$ such that, for every $n\in\mathbb N$, $\frac{\left|A\right|}{\left|B\right|}\geq C$.
\end{prop}
\begin{cor}\label{gap}
The lengths of the gaps of the dynamical partition $\mathscr P_{n}$ tend to zero at least exponentially fast with $n$.
\end{cor}
\begin{lemma}\label{cielo}
The ratio 
\begin{displaymath}
\frac{\left|\underline{-q_n+iq_{n-1}}\right|}{\left|\left[\underline{-q_n+iq_{n-1}},\underline{0}\right)\right|}
\end{displaymath}
is bounded away from zero by a uniform constant for all $i=0,\dots,a_n$.
\end{lemma}
For the proof of Proposition \ref{figo}, Corollary \ref{gap} and Lemma \ref{cielo} see \cite{a}.
\begin{lemma}\label{brava!}
For every $i\in\{0,\dots,a_n-1\}$ the ratio
\begin{displaymath}
\omega_i=\frac{\left|\left(\underline{-q_n+(i+1)q_{n-1}},\underline{-q_n+iq_{n-1}}\right)\right|}{\left|\underline{-q_n+iq_{n-1}}\right|}
\end{displaymath}
 is uniformly comparable to $\alpha_{n-1}$.
\end{lemma}
\begin{proof}
For every $i\in\{ 1,\dots,a_n-1\}$, we apply Proposition \ref{Koebe liv} to
\begin{itemize}
\item[-] $T=\left[\underline{-q_{n}+(i+1)q_{n-1}},\underline{-q_{n}+(i-1)q_{n-1}}\right]$,
\item[-] $J=\left(\underline{-q_{n}+(i+1)q_{n-1}},\underline{-q_{n}+(i-1)q_{n-1}}\right)$
\item[-] $f^{q_{n}-(i+1)q_{n-1}}$.
\end{itemize}

Observe that the hypotheses  are satisfied:
\begin{itemize}
\item[-] $T$ doesn't contain any preimage of the flat interval of the type $f^{-j}$ with $j< q_{n}-(i+1)q_{n-1} $, so $f^{q_{n}-(i+1)q_{n-1}}$ is a diffeomorphism on $T$; 
\item[-] the set $\bigcup_{j=0}^{q_{n}-(i+1)q_{n-1}-1}f^j\left(J\right)$ covers each point of $\mathbb S^1$ at most twice;
\item[-] if $n$ is large enough, by Lemma \ref{cielo} there exists a constant $K$ such that $\left|f^{q_{n}-(i+1)q_{n-1}}(J)\right|=\left|\left(\underline{0},\underline{-2q_{n-1}}\right)\right|\leq K\left |\underline{-2q_{n-1}}\right|$.
\end{itemize}

So, by Proposition \ref{Koebe liv} and Proposition \ref{figo}, $\omega_i$ is comparable to $\alpha_{n-1}$. \\

For $i=0$ we apply Proposition \ref{Koebe liv} to
\begin{itemize}
\item[-] $T=\left[\underline{-q_{n}+q_{n-1}+1},\underline{-q_{n}+1}\right]$,
\item[-] $J=\left(\underline{-q_{n}+q_{n-1}+1},\underline{-q_{n}+1}\right)$
\item[-] $f^{q_{n}-q_{n-1}-1}$.
\end{itemize}
As before, the hypotheses are satisfied (the intervals $f^{j}\left(J\right)$ for \\
$j\in\left\{1,\dots,q_{n}-q_{n-1}-2\right\}$ are disjoint).\\
So, for $n$ large
\begin{displaymath}
\omega_0^l=\frac{\left|\left(\underline{-q_n+q_{n-1}+1},\underline{-q_n+1}\right)\right|}{\left|\underline{-q_n+1}\right|}
\end{displaymath}
is uniformly comparable to $\alpha_{n-1}$.\\
This concludes the proof.
\end{proof}

\section{Proof of Theorem $\ref{1}$}

In the proof of Theorem 2 in \cite{a} it is shown that, if the critical exponent of the map $f$ is less than or equal to $2$, then there exists $0<\lambda<1$ such that, for every $n$ large enough
\begin{equation}\label{alpha}
\alpha_n=\frac{|(\underline{-q_n},\underline{0})|}{|[\underline{-q_n},\underline{0})|}\leq \lambda^n.
\end{equation}
\begin{lemma}\label{sole}
There is a positive constant $C$ such that  each gap $A^{n-2}$ of $\mathscr P_{n-2}$ is split into at least two gaps $A_i^{n}$ of $\mathscr P_{n}$ whose lengths satisfy $|A_i^{n}|<C\lambda^{n-1}|A^{n-2}|$.
\end{lemma}
\begin{proof}
Because of the symmetry, it is enough to suppose that $n$ is even,that is $\underline{-q_n}$ is to the left of $\underline{0}$.\\

Let $A^{n-2}$ be a long gap of  $\mathscr P_{n-2}$;
\begin{displaymath}
 A^{n-2}=\left(\underline{-q_{n-2}},\underline{0}\right).
\end{displaymath}
As explained by the formula (\ref{div}), $A^{n-2}$ is subdivided into $a_n$ long gaps and one short gap of the dynamical partition $\mathscr P_{n-1}$. In our case these gaps are 
\begin{displaymath}
A_i^{n}=(\underline{-q_{n}+(i+1)q_{n-1}},\underline{-q_{n}+iq_{n-1}})
\end{displaymath}
for $i\in\{0,\dots,a_{n}-1\}$.\\

By Lemma \ref{brava!} and  the inequality (\ref{alpha}), there exists a constant $C$ such that, for every $i\in\{0,\dots, a_n-1\}$
\begin{displaymath}
\frac{\left|A_i^n\right|}{\left|A^{n-2}\right|}<\frac{\left|\left(\underline{-q_{n}+(i+1)q_{n-1}},\underline{-q_{n}+iq_{n-1}}\right)\right|}{\left|\underline{-q_{n}+iq_{n-1}}\right|}\leq C\alpha_{n-1}\leq C\lambda^{n-1}.
\end{displaymath} 

Since the shorts gaps of $\mathscr P_{n-2}$ become the longs gaps of the partition $\mathscr P_{n-1}$, we can use the estimates for long gaps of the partition $\mathscr P_{n-1}$ to conclude the proof.
\end{proof}
\paragraph{Proof of Theorem \ref{1}}

Observe that $\Omega_{n-2}=\bigcup_{A^{n-2}\in\mathscr P_{n-2}}A^{n-2}$ and  $\Omega_{n}=\bigcup_{A^{n}\in\mathscr P_{n}}A^{n}$ are two covers of $K$. We analyze the relations between them.\\
By Lemma \ref{sole} and the formula (\ref{div}), we can deduce that:
\begin{enumerate}
\item
 Each interval $A^{n-2}$ of $\Omega_{n-2}$ is split into at least two intervals $A_i^{n}$ of $\Omega_n$ whose lengths satisfy $|A_i^{n}|<C\lambda^{n-1}|A^{n-2}|$.
\item
 Each interval $A^{n-2}$ is split into at most $a_n\left(a_{n+1}+1\right)+1$ sub-intervals.
\end{enumerate}
Consequently,
\begin{displaymath}
\sum_{A^n\in\mathscr P_n}|A^{n}|^s\leq C^s(a_n\left(a_{n+1}+1\right)+1)\lambda^{(n-1)s}\sum_{A^{n-2}\in\mathscr P_{n-2}}|A^{n-2}|^s.
\end{displaymath}
Since the maximum diameter of the covering $\{A^n\}$ goes to zero as $n$ increases, the Hausdorff s-measure of $K$ must be zero for every $s>0$.
This implies the theorem.

\section{Proof of Theorem $\ref{2}$}
In this section we will assume that the critical exponent $\left(l,l\right)$ of $f$ is strictly greater than $2$.
\begin{lemma}\label{aiuto}
Two adjacent gaps of the dynamical partition $\mathscr P_{n}$ which are contained in the same gap of  $\mathscr P_{n-1}$ are comparable.
\end{lemma}
\begin{proof}
Because of the symmetry, without loss of generality, we may assume that $n$ is even; so $\underline{-q_n}$ is to the left of the flat interval.\newline

For every $i\in\{0,\dots,a_{n+1}-1\}$ let
\begin{displaymath}
A_i=\underline{-q_{n+1}+iq_{n}}
\end{displaymath}
and let
\begin{displaymath}
B_i=(\underline{-q_{n+1}+iq_{n}}, \underline{-q_{n+1}+(i+1)q_{n}}).
\end{displaymath}

By Lemma \ref{brava!}, there exists a constant $C$ such that, for every $i\in\{0,\dots,a_{n+1}-1\}$
\begin{equation}\label{sonno}
\frac{\left|B_i\right|}{\left|A_i\right|}\geq C\alpha_n>C\tau_{n}
\end{equation}
 is bounded away from zero by Theorem \ref{sw}.\\
So, by the formula (\ref{sonno}) and Proposition \ref{figo}, there exists a constant $K>0$ such that
\begin{displaymath}
\frac{|B_{i-1}|}{|B_i|}\leq K.
\end{displaymath}

In order to compare $B_{i-1}$ and $A_i$, we observe that, for every $i\in\{1,\dots,a_{n+1}-1\}$ 
\begin{equation}\label{swi}
\frac{\left|A_i\right|}{\left|\underline{-q_{n+1}}\right|}\leq K_1\frac{\left|A_i\right|}{\left|\left(\underline{0},\underline{-q_{n+1}}\right)\right|}\leq K_1\frac{\left|\left(\underline{0}, A_i\right]\right|}{\left|\left(\underline{0},\underline{-q_{n+1}}\right)\right|}\leq K_1\frac{\left|\left(\underline{0}, \underline{-q_{n-1}}\right)\right|}{\left|\left(\underline{0},\underline{-q_{n+1}}\right)\right|}\leq K_2
\end{equation}
where the first inequality follows from the Proposition \ref{figo} and the last one from Theorem \ref{sw}.\\
Therefore, using the inequality (\ref{sonno}), Proposition \ref{figo} and the inequality (\ref{swi}), we have that
\begin{equation}\label{sic}
\frac{\left|B_{i-1}\right|}{\left|A_i\right|}\geq C\frac{\left|A_{i-1}\right|}{\left|A_i\right|}\geq C_1\frac{\left|B_{i-2}\right|}{\left|A_i\right|}\geq\dots\geq C_i\frac{\left|\underline{-q_{n+1}}\right|}{\left|A_i\right|}\geq\frac{C_i}{K_2}.
\end{equation}
Moreover, $|\left(\underline{0},\underline{-q_n}\right)|$ is bigger than a uniform constant multiple of $|\underline{-q_n}|$ as the sequence $\tau_n$ is bounded away from zero for $l>2$.\\
Finally, by the formula (\ref{sic}) and Proposition \ref{figo}, there exists a constant $K_1$ such that
\begin{displaymath}
\frac{|B_{i-1}|}{|B_i|}\geq K_1.
\end{displaymath}
So, $B_{i-1}$ and $B_{i}$ are comparable.

\end{proof}
\begin{lemma}\label{big}
There exists a constant $C$ such that
\begin{displaymath}
\left|\left(\underline{q_n},\underline{-q_{n+1}+q_n}\right)\right|\geq C\left|\left(\underline{-q_{n+1}},\underline{-q_{n+1}+q_n}\right)\right|
\end{displaymath}
\end{lemma}
\begin{proof}
Apply Proposition \ref{Koebe liv} to
\begin{itemize}
\item[-]$T=\left[\underline{-q_{n+1}}, \underline{-q_{n+1}+q_n}\right]$
\item[-]$J=\left(\underline{-q_{n+1}}, \underline{-q_{n+1}+q_n}\right)$
\item[-]$f^{q_{n+1}-q_n}$.
\end{itemize}
Also in this case the hypotheses are satisfied:
\begin{itemize}
\item[-] $f^{q_{n+1}-q_n}$ is a diffeomorphism on $T$;
\item[-] the intervals $f^i\left(J\right)$ for $i\in\left\{1,\dots,q_{n+1}-q_n-1\right\}$ are disjoint;
\item[-] for $n$ large enough, by Proposition \ref{figo} there exists a constant $K$ such that 
\begin{displaymath}
\left|f^{q_{n+1}-q_n}\left(J\right)\right|=\left|\left(\underline{-q_n},\underline{0}\right)\right|\leq K \left|\underline{-q_{n}}\right|=K dist\left(f^{q_{n+1}-q_n}\left(J\right),\partial f^{q_{n+1}-q_n}\left(T\right)\right).
\end{displaymath}
\end{itemize}
So, there exists a constant $C$ such that
\begin{displaymath}
\frac{\left|\left(\underline{q_n},\underline{-q_{n+1}+q_n}\right)\right|}{\left|\left(\underline{-q_{n+1}},\underline{-q_{n+1}+q_n}\right)\right|}\geq C\frac{\left|\left(\underline{q_{n+1}},\underline{0}\right)\right|}{\left|\left(\underline{-q_{n}},\underline{0}\right)\right|} > C\frac{\left|\left(\underline{-q_{n+2}},\underline{0}\right)\right|}{\left|\left(\underline{-q_{n}},\underline{0}\right)\right|}
\end{displaymath}
which is bounded away from zero by Theorem \ref{sw}.\\
\end{proof}
\begin{lemma}\label{help}
Two gaps of the same dynamical partition which are adjacent to the flat interval are comparable.
\end{lemma} 
\begin{proof}
Because of the symmetry, it is enough to suppose that $n$ is even.\newline

Let $A=\left(\underline{-q_n},\underline{0}\right)$ and let $B=\left(\underline{0},\underline{-q_{n+1}}\right)$; the initial situation is explained in the following figure.
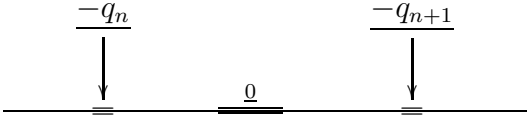
\begin{center}
\xymatrix{
& {\underline{-q_{n}}} \ar[d]& && {\underline{-q_{n+1}}}\ar[d]
 \\
\ar@{-}[rrrrr]&{=}&\ar@{=}[r]^{\underline{0}}& &{=}&  }
\end{center}

 The idea is to apply Proposition \ref{Koebe liv} to 
\begin{itemize}
\item[-]$T=\left[\underline{-q_{n}+1}, \underline{-q_{n+1}+1}\right]$
\item[-]$J=\left(\underline{-q_{n}+1}, \underline{-q_{n+1}+1}\right)$
\item[-]$f^{q_{n}-1}$.
\end{itemize}
So, there exists a constant $C$ such that
\begin{equation}\label{cuore}
\frac{\left|f\left(B\right)\right|}{\left|f\left(A\right)\right|}\geq C\frac{\left|\left(\underline{q_{n}},\underline{-q_{n+1}+q_n}\right)\right|}{\left|\left(\underline{0},\underline{q_n}\right)\right|}\geq C\frac{\left|\left(\underline{q_{n}},\underline{-q_{n+1}+q_n}\right)\right|}{\left|\left(\underline{0},\underline{-q_{n-1}}\right)\right|}.
\end{equation}
 By Theorem \ref{sw} and by Lemma \ref{big} there exist two constants $K$ and $K_1$ such that
\begin{displaymath}
\frac{\left|\left(\underline{q_{n}},\underline{-q_{n+1}+q_n}\right)\right|}{\left|\left(\underline{0},\underline{-q_{n-1}}\right)\right|}\geq K\frac{\left|\left(\underline{q_{n}},\underline{-q_{n+1}+q_n}\right)\right|}{\left|\left(\underline{0},\underline{-q_{n+1}}\right)\right|}\geq K_1 K\frac{\left|\left(\underline{-q_{n+1}},\underline{-q_{n+1}+q_n}\right)\right|}{\left|\left(\underline{0},\underline{-q_{n+1}}\right)\right|}.
\end{displaymath}
Observe that the last intervals are contained in the same gap of the partition $n$; so by Lemma \ref{aiuto} they are comparable.\\

Similarly, there exists $C>0$
\begin{equation}\label{ast}
\frac{|f(A)|}{|f(B)|}\geq C\frac{\left|\left(\underline{0},\underline{q_n}\right)\right|}{\left|\left(\underline{q_{n}},\underline{-q_{n+1}+q_n}\right)\right|}\geq C\frac{\left|\left(\underline{0},\underline{q_{n}}\right)\right|}{\left|\left(\underline{0},\underline{q_{n-2}}\right)\right|}\geq C\tau_n.
\end{equation}
which is bounded away from zero by Theorem 2 in \cite{a}.\newline
So, $f(A)$ and $f(B)$ are comparable.\\
Moreover, we recall that near to the flat interval $f(x)=x^l$ hence also $A$ and $B$ are comparable.


\end{proof}
\begin{prop}\label{flat}
Two adjacent gaps of the dynamical partition $\mathscr P_{n}$ are comparable. 
\end{prop}
\begin{proof}
By the symmetry, without loss of generality, we can suppose that $n$ is even.\\


Case 1. For $0\leq i < q_n$ we have\\
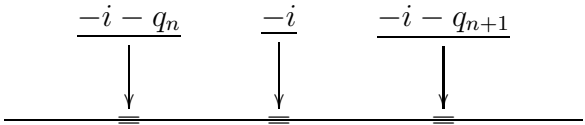
\xymatrix{& {\underline{-i-q_{n}}}\ar[d]&{\underline{-i}}\ar[d]& {\underline{-i-q_{n+1}}}\ar[d]  \\
\ar@{-}[rrr]&{=}&{=}&{=}&\ar@{-}[ll]}
\newline
\\
Let $A=(\underline{-i-q_n},\underline{-i})$ and let $B=(\underline{-i},\underline{-i-q_{n+1}})$.\\

For $i=0$ we find the two gaps adjacent to the flat interval which are uniformly comparable, see Lemma \ref{help}.\\

For $i\neq 0$ apply Proposition \ref{Koebe liv} to
\begin{itemize}
\item[-]$T=\left[\underline{-i-q_{n}}, \underline{-i-q_{n+1}}\right]$,
\item[-]$J=\left(\underline{-i-q_{n}}, \underline{-i-q_{n+1}}\right)$,
\item[-]$f^{i}$.
\end{itemize}

Hence, by Proposition \ref{Koebe liv} and the inequality (\ref{cuore}) there exist two constants $C$ and $C_1$ such that
\begin{displaymath}
\frac{\left|A\right|}{\left|B\right|}\leq C\frac{\left|\left(\underline{-q_{n}+1},\underline{1}\right)\right|}{\left|\left(\underline{1},\underline{-q_{n+1}+1}\right)\right|}\leq CC_1.
\end{displaymath}
Similarly, Proposition \ref{Koebe liv} and the inequality (\ref{ast}) imply that
\begin{displaymath}
\frac{\left|B\right|}{\left|A\right|}\leq C\frac{\left|\left(\underline{1},\underline{-q_{n+1}+1}\right)\right|}{\left|\left(\underline{-q_{n}+1},\underline{1}\right)\right|}\leq CC_2.
\end{displaymath}
where $C$ and $C_2$ are uniform constants.
\\

Case 2. For $q_n\leq i < q_{n+1}$ we have\\
\xymatrix{& {\underline{-i-q_{n}}}\ar[d]&{\underline{-i}}\ar[d]& {\underline{-i+q_{n}}}\ar[d]  \\
\ar@{-}[rrr]&{=}&{=}&{=}&\ar@{-}[ll]}
\newline
\\
Let $A=(\underline{-i-q_n},\underline{-i})$ and let $B=(\underline{-i},\underline{-i+q_{n}})$.\\

We apply Proposition \ref{Koebe liv} to
\begin{itemize}
\item[-]$T=\left[\underline{-i-q_{n}}, \underline{-i+q_{n}}\right]$,
\item[-]$J=\left(\underline{-i-q_{n}}, \underline{-i+q_{n}}\right)$,
\item[-]$f^{i-q_{n}}$.
\end{itemize}
There exists a constant $C$ such that
\begin{displaymath}
\frac{\left|A\right|}{\left|B\right|}\leq C\frac{\left|\left(\underline{-2q_{n}},\underline{-q_n}\right)\right|}{\left|\left(\underline{-q_n},\underline{0}\right)\right|}\leq  C\frac{\left|\left(\underline{-q_{n-2}},\underline{0}\right)\right|}{\left|\left(\underline{-q_n},\underline{0}\right)\right|}\leq CC_1.
\end{displaymath}
$C_1$ is a constant that comes from Theorem \ref{sw}.\\

Moreover by  Lemma \ref{aiuto} and Lemma \ref{cielo} there exist two constants $K_1$ and $K_2$ such that 
\begin{displaymath}
\frac{\left|B\right|}{\left|A\right|}\leq C\frac{\left|\left(\underline{-q_n},\underline{0}\right)\right|}{\left|\left(\underline{-2q_{n}},\underline{-q_n}\right)\right|}\leq C K_1\frac{\left|\left(\underline{-q_n},\underline{0}\right)\right|}{\left|\underline{-2q_{n}}\right|}\leq C K_1\frac{\left|\left[\underline{-2q_n},\underline{0}\right)\right|}{\left|\underline{-2q_{n}}\right|}\leq C K_1K_2 .
\end{displaymath}
\\

Case 3. For $q_{n+1}\leq i < q_n+q_{n+1}$ we have\\
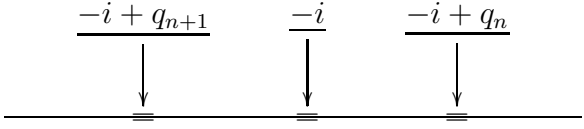
\xymatrix{& {\underline{-i+q_{n+1}}}\ar[d]&{\underline{-i}}\ar[d]& {\underline{-i+q_{n}}}\ar[d]  \\
\ar@{-}[rrr]&{=}&{=}&{=}&\ar@{-}[ll]}
\newline
Let $A=\left(\underline{-i+q_{n+1}},\underline{-i}\right)$ and let $B=(\underline{-i},\underline{-i+q_{n}})$. \\

We apply Proposition \ref{Koebe liv} to
\begin{itemize}
\item[-]$T=\left[\underline{-i+q_{n+1}}, \underline{-i+q_{n}}\right]$,
\item[-]$J=\left(\underline{-i+q_{n+1}}, \underline{-i+q_{n}}\right)$,
\item[-]$f^{i-q_{n+1}}$.
\end{itemize}
There exists a constant $C$ such that
\begin{displaymath}
\frac{\left|A\right|}{\left|B\right|}\leq C\frac{\left|\left(\underline{0},\underline{-q_{n+1}}\right)\right|}{\left|\left(\underline{-q_{n+1}},\underline{-q_{n+1}+q_n}\right)\right|}.
\end{displaymath}
Since $\left(\underline{0},\underline{-q_{n+1}}\right)$ and $\left(\underline{-q_{n+1}},\underline{-q_{n+1}+q_n}\right)$ are contained in the same gap of the partition $n-1$, we can use Lemma \ref{aiuto} to obtain a uniform upper bound.\\

Similarly, by Proposition \ref{Koebe liv} and Lemma \ref{aiuto}, there exist two constants $C$ and $C_1$ such that
\begin{displaymath}
\frac{\left|B\right|}{\left|A\right|}\leq C\frac{\left|\left(\underline{-q_{n+1}},\underline{-q_{n+1}+q_n}\right)\right|}{\left|\left(\underline{0},\underline{-q_{n+1}}\right)\right|}\leq CC_1.
\end{displaymath}

All the possibilities have been analysed, so the proof is complete.

\end{proof}
\begin{proof}Proof of Theorem $\ref{2}$.\newline
Let $\mathscr A_{n}$ be the algebra generated by the set of all gaps belonging to the partition $\mathscr P_{n}$.
We define a probability measure $\mu$ on $\mathscr A_{n}$.\newline
For every gap $I\in\mathscr P_{1}$ set $\mu(I)=\frac{1}{\#\mathscr P_{1}}$.
Suppose $\mu$ is already constructed on $\mathscr A_{n}$ and we fix $\mu$ on $\mathscr A_{n+1}$.\newline
If $I_{i}^{n}$ is a long gap of $\mathscr P_{n}$ then
\begin{displaymath}
 I_{i}^{n}=\bigcup_{j=1}^{a_{n+2}}f^{-i-q_n-jq_{n+1}}(U)\cup\bigcup_{j=0}^{a_{n+2}-1}I_{i+q_n+jq_{n+1}}^{n+1}\cup I_{i}^{n+2}
\end{displaymath}
and we set, for every $j\in\{0,\dots, a_{n+2}-1\}$,
\begin{displaymath}
\mu(I_{i+q_n+jq_{n+1}}^{n+1})=\frac{\mu(I_{i}^{n})}{2a_{n+2}}
\end{displaymath}
 and
\begin{displaymath}
\mu(I_{i}^{n+2})=\frac{\mu(I_{i}^{n})}{2}.
\end{displaymath}
If  $I_{i}^{n+1}$ is a short gap of $\mathscr P_{n}$ then $\mu\left(I_{i}^{n+1}\right)$ is already defined.\newline
By Carath\'eodory's extension theorem, there exists a measure (that we will call again $\mu$) which extends $\mu$ on $\sigma(\mathscr A_{1},\mathscr A_{2},..)$.\\

By the definition, $\mu\left(I_{i}^{n}\right)\leq\frac{1}{2^{\frac{n}{2}}}$. 
Proposition \ref{flat} and Corollary \ref{gap} imply that there exist $\lambda_1,\lambda_2\in(0,1)$ such that all gaps satisfy $\lambda_{1}^{n}\leq\left|I_{i}^{n}\right|\leq\lambda_{2}^{n}$. \\
Therefore,
\begin{displaymath}
\mu\left(I_{i}^{n}\right)\leq\left|I_{i}^{n}\right|^{\alpha}
\end{displaymath}
with $\alpha=\log_{\lambda_1}\frac{1}{\sqrt{2}}$.\\

Let $I$ be an arbitrary interval and let $I_{i}^{n+1}$ be a gap contained in $I$ with $n$ as small as possible. Then $I$ is covered by at most two gaps of the $n^{th}$ partition $I_{j}^{n}$ and $I_{j'}^{n}$ and by preimages of $U$ (which are of $\mu$-measure zero). \\
By the definition of $\mu$
\begin{equation}\label{mis}
\mu\left(I\right)\leq\mu\left(I_{j}^{n}\right)+\mu\left(I_{j'}^{n}\right)\leq C\mu\left(I_{j}^{n}\right)\leq C'\mu\left(I_{i}^{n+1}\right)\leq C' \left|I_{i}^{n+1}\right|^{\alpha}\leq C'\left|I\right|^{\alpha}
\end{equation}

Finally, let $\mathscr{K}$ be an $\epsilon$-cover of the non-wandering set $K$ and let $0<\alpha<1$.\\
By the inequality (\ref{mis}),
\begin{displaymath}
\sum_{I\in\mathscr K}\left|I\right|^{\alpha}\geq\frac{1}{C'}\sum_{I\in\mathscr K}\nu\left(I\right)\geq \frac{1}{C'}\nu\left(K\right)=\frac{1}{C'}.
\end{displaymath}
This establishes the theorem.
\end{proof}
\section{Applications: Cherry Flows}\label{Cherry flow}
\subsection{Basic Definitions}
Let $X$ be a $C^{\infty}$ vector field on the torus $T^2$. Denote the flow through a point $x$ by $t\rightarrow X_t(x)$ and by $Sing(X)$ the set of singularities of $X$. The positive and negative \textit{semi-trajectories} are the sets
\begin{displaymath}
l^{+}(x)=\{ X_t(x) : t\geq 0\} 
\end{displaymath}
\begin{displaymath}
l^{-}(x)=\{ X_t(x) : t\leq 0\}.
\end{displaymath}
The set $l(x)=l^{+}(x)\cup l^{-}(x)$ is called the \textit{trajectory} through the point $x$. If $l(x)=x$ then $x$ is called a \textit{fixed point} or a \textit{singularity}. If $x\in T^2$ is not a fixed point then it is called a \textit{regular} point and $l(x)$ will be called a \textit{one-dimensional} trajectory. If $l(x)$ is homeomorphic to a circle $\mathbb S^1$, then $l(x)$ is a \textit{closed} trajectory or a \textit{periodic} trajectory. The trajectory $l(x)$ is called a \textit{non-closed} trajectory if $l(x)$ is neither a fixed point nor a periodic trajectory. The $\omega$-\textit{limit set} of the positive semi-trajectory $l^{+}(x)$ is the set
\begin{displaymath}
\omega[l^{+}(x)]=\left\{y;\exists t_n\rightarrow\infty \textmd{ with } X_{t_n}(x)\rightarrow y \right\},
\end{displaymath}
and $\alpha$-\textit{limit set} of the negative  semi-trajectory $l^{-}(x)$ is
\begin{displaymath}
\alpha[l^{-}(x)]=\left\{y;\exists t_n\rightarrow\infty \textmd{ with } X_{-t_n}(x)\rightarrow y \right\}.
\end{displaymath}
The $\omega$-limit set ($\alpha$-limit set) of any positive (negative) semi-trajectory of the trajectory $l$ is called $\omega$-limit set $\omega(l)$ of $l$ ($\alpha$-limit set $\alpha(l)$ of $l$). The union of $\omega$-limit sets of all trajectories of the flow $X_t$ is called the $\omega$-limit set $\omega (X_t)$ of $X_t$ (analogously for the $\alpha$-limit set $\alpha (X_t)$ of $X_t$ ). The union $\lim(X_t)=\omega(X_t)\cup\alpha(X_t)$ is said to be the \textit{limit set} of $X_t$.\\
The trajectory is $\omega$-recurrent ($\alpha$-recurrent), if it is contained in its $\omega$-limit set ($\alpha$-limit set). The trajectory is \textit{recurrent} if it is both $\omega-$recurrent and $\alpha-$recurrent. A recurrent trajectory is \textit{non-trivial} if it is neither a fixed point nor a periodic trajectory.\\
Let $S$ be a subset of $T^2$. We denote by $W^{s}(S)$ the set of points in $T^2$ that have $S$ as $\omega$-limit (it is called the \textit{stable manifold} of $S$) and  by $W^{u}(S)$ the set of points that have $S$ as $\alpha$-limit (it is called the \textit{unstable manifold} of $S$).
\begin{defin}
A Cherry field is a $\mathfrak C^{\infty}$ vector field on the torus $T^2$ without closed trajectories which has exactly two singularities, a sink and a saddle, both hyperbolic.
\end{defin}
The first example of such a field was given by Cherry in $1938$.
\begin{defin}
The Poincar\'e global section is a transversal, simple, closed $C^{\infty}$ curve $\Sigma$ which intersects every one-dimensional trajectory of the flow.
\end{defin}
\begin{fact}
It is proved in \cite{cherry} that every Cherry flow admits a Poincar\`e global section.
\end{fact}
\begin{defin}
Let $l$ be a non-trivial recurrent trajectory. Then the closure $\overline{l}$ of $l$ is called a quasi-minimal set.
\end{defin}
The following fact was proved in \cite{cherry}.
\begin{fact}\label{prod}
Every Cherry field has a quasi-minimal set which is locally homeomorphic to the Cartesian product of a Cantor set and a segment.
\end{fact}
\subsection{First Return Map}
Let $X$ be a Cherry field with a global Poincar\`e section $\Sigma$. Notice that $T^2\setminus\Sigma$ is $C^{\infty}$-equivalent to an annulus $\mathbb S^1\times(0,1)$ and we can write $T^2\cong\mathbb S^1\times\left[0,1\right]/\sim$, where $\left(s,0\right)\sim\left(s,1\right)$. Consider $X$ as a flow on $T^2\cong\mathbb S^1\times\left[0,1\right]$ where we identify $\mathbb S^1\times\left\{0\right\}$ and $\mathbb S^1\times\left\{1\right\}$.\newline
After this change of coordinates, the resulting field is a Cherry field.\\
We denote by $W$ the set of points $x\in\mathbb S^1\times\left\{0\right\}$ with $X_t(x)\in\mathbb S^1\times\left\{1\right\}$ for some $t>0$. For $x\in W$, let $t(x)$ be the minimal $t>0$ such that $X_{t(x)}(x)\in\mathbb S^1\times\left\{1\right\}$ and define the map $f:W\rightarrow\mathbb S^1$ by $f(x)=X_{t(x)}(x)$. This map is called the \textit{first return map} to $\mathbb S^1$.\newline
We define $f$ on $U=(\mathbb S^1\setminus W)$ as $f(U)=v$ where $v=W^u(S)\cap\mathbb S^1\times\left\{1\right\}$.\\
For every point $x\in(\mathbb S^1\setminus U)\times\left\{0\right\}$ near $\partial U$, there exists $t>0$ such that $X_t(x)$ intersects $\mathbb S^1\times\left\{1\right\}$ near $v$ (for more details see the section $2$ of the chapter $7$ in \cite{aran}).  In particular $\lim_{x\rightarrow\partial U, x\notin U}f(x)$ consists of a single point.\\

So, we have a function $f$ of the circle, which is  everywhere continuous and $C^{\infty}$ outside the boundary points of $\mathbb S^1\setminus U$.
\subsection{Properties of the First Return Map}
Let $X$ be a Cherry vector field, with a saddle point that has eigenvalues $\lambda_{1}>0>\lambda_{2}$. Let $\mathbb S^1$ be a global Poincar\'e section and let $f$ be the first return map to $\mathbb S^1$.\newline
Then: 
\begin{enumerate}
\item $f$ is order preserving.
\item It is constant on an interval $U$.
\item The restriction of $f$ to $\mathbb S^1\setminus{\overline{U}}$ is a $C^{\infty}$-diffeomorphism onto its image.\newline
\item Let $\pi\left(a,b\right)=U$. On some right-sided neighborhood of $b$, $f$ can be represented as
\begin{displaymath}
h_r\left(\left(x-b\right)^{\alpha}\right),
\end{displaymath}
where $\alpha=\frac{\left|\lambda_2\right|}{\lambda_1}$ and $h_r$ is a $C^{\infty}$-diffeomorphism on a neighborhood of $b$.
Analogously, on a left-sided neighborhood of $a$, $f$ is 
\begin{displaymath}
h_l\left(\left(a-x\right)^{\alpha}\right).
\end{displaymath}
\item Since $X$ has no periodic trajectories, the rotation number of $f$, $\rho(f)$, is irrational.
\end{enumerate}
These are well-known properties of $f$, the reader can consult for example $\cite{aran}$ or $\cite{cherry}$.\\

The first return map $f$ belongs to the class $\mathscr{L}$ from Section \ref{class}.
\begin{lemma}
There exists a Cherry vector field with first return map which has rotation number of bounded type.
\end{lemma}
\begin{proof}
Let $X$ be a Cherry vector field with a global Poincar\'e section $\mathbb S^1$.\\
For every $\epsilon> 0$,  we define $Y_{\epsilon}$ as a constant horizontal vector field with all vectors of the same length $\epsilon$.
Let $X_{\epsilon}=X+Y_{\epsilon}$ and denote by $f_{\epsilon}$ the first return map of $X_{\epsilon}$ to $\mathbb S^1$.\\
The family $(f_{\epsilon})$ is increasing and hence $\rho_{\epsilon}=\rho(f_{\epsilon})$ is an increasing and continuous function. 
Since $\rho_0$ is irrational, $\rho_{\epsilon}$ is strictly increasing at zero and hence, there is a $\epsilon_{0}$ so that $\rho_{\epsilon_{0}}$ is an irrational number of bounded type.
\end{proof}
\subsection{ Proof of Theorem \ref{3}}
\begin{proof}
We denote by $\HD(M)$ the Hausdorff dimension of a set $M$ and by $Q$ the quasi-minimal set of $X$. By Fact \ref{prod}, in a small neighborhood of the Poincar\'e global section, $Q$ is equivalent, by a $C^2$ diffeomorphism, to $I\times \Omega$ with $\Omega$ the non-wandering set. \\
So, by Theorem \ref{2} and the product formula, Theorem 8.10 in \cite{ka},
\begin{displaymath}
\HD(Q)\geqq \HD(I\times \Omega)\geqq \HD(I)+\HD(\Omega)> 1.
\end{displaymath}

\end{proof}
\appendix
\section{Appendix}
In the following we explain how to remove the hypothesis on the Schwarzian derivative in the first claim of Theorem 2 in \cite{a}. \\

The negative Schwarzian is used only to prove the expansion property of the cross-ratio \textbf{Poin} which is defined on ordered quadruples $a<b<c<d$ by
\begin{displaymath}
\textbf{Poin}\left(a,b,c,d\right):=\frac{\left|d-a\right|\left|b-c\right|}{\left|c-a\right|\left|d-b\right|}.
\end{displaymath}
Diffeomorphisms with negative Schwarzian derivative increase the cross-ratio \textbf{Poin}:
\begin{displaymath}
\textbf{Poin}\left(f(a),f(b),f(c),f(d)\right)>\textbf{Poin}\left(a,b,c,d\right).
\end{displaymath}\\
In general, without the assumption of negative Schwarzian, the following holds.
\begin{theo}\label{point}
Let $f$ be a $C^2$ map with no flat critical points. There exists a bounded increasing function $\sigma:\left[0,\infty\right)\to\mathbb R_+$ with $\sigma(t)\to 0$ as $t\to 0$ with the following property. Let $\left[b,c\right]\subset\left[a,d\right]$ be intervals such that $f^n_{|\left[a,d\right]}$ is a diffeomorphism. Then
\begin{displaymath}
\textbf{Poin}\left(f^n(a),f^n(b),f^n(c),f^n(d)\right)\geqq exp\{-\sigma(\tau)\sum_{i=0}^{n-1}|f^{i}(\left[a,b\right))|\}\textbf{Poin}\left(a, b, c, d\right),
\end{displaymath}
where $\tau=\max_{i=0,\dots,n-1}|f^{i}(\left(c,d\right])|$.
\end{theo}

We recall the notation used in \cite{a}, let
\begin{displaymath}
\alpha_n=\frac{\left|\left(\underline{-q_n},\underline{0}\right)\right|}{\left|\left[\underline{-q_n},\underline{0}\right)\right|},
\end{displaymath}
\begin{displaymath}
\sigma_n=\frac{\left|\left(\underline{0},\underline{q_n}\right)\right|}{\left|\left(\underline{0},\underline{q_{n-1}}\right)\right|}
\end{displaymath}
and 
\begin{displaymath}
s_n=\frac{\left|\left[\underline{-q_{n-2}},\underline{0}\right]\right|}{\left|\underline{0}\right|}.
\end{displaymath}
\begin{theo}
For $n$ sufficiently large the following inequality holds:
\begin{equation}\label{alfa}
\left(\alpha_n\right)^l\leq K_{0,n} K_{1,n},\dots, K_{{a_n-1},n} C_n \tilde{M_n}\left(l\right)\alpha_{n-2}^2,
\end{equation}
where for every $i\in{0,\dots,a_n-1}$ if we denote by
\begin{displaymath}
\tau_{i,n}=\max_{j=0,\dots,q_{n-1}-2}\left|f^{j}\left(\left(\underline{iq_{n-1}+1},\underline{-q_{n-1}+1}\right]\right)\right|
\end{displaymath}
and by
\begin{displaymath}
\rho_n=\max_{j=0,\dots,q_{n-2}-2}\left|f^{j}\left(\left(\underline{a_nq_{n-1}+1},\underline{1}\right]\right)\right|
\end{displaymath}
then
\begin{equation}\label{dom}
K_{i,n}=e^{\sigma\left(\tau_{i,n}\right)\sum_{j=0}^{q_{n-1}-2}|f^{j}(\underline{-q_n+iq_{n-1}+1})|},
\end{equation}
\begin{displaymath}
C_n=e^{\sigma\left(\rho_n\right)\sum_{j=0}^{q_{n-2}-2}|f^{j}(\underline{-q_{n-2}+1})|}
\end{displaymath}
and
\begin{displaymath}
\tilde{M_n}\left(l\right)=s_{n-1}^2\cdot\frac{2}{l}\cdot\left(\frac{1}{1+\sqrt{1-\frac{2(l-1)}{l}C_ns_{n-1}\alpha_{n-1}}}\right)\cdot\frac{1}{1-\alpha_{n-2}}\cdot\frac{\sigma_n}{\sigma_{n-2}}.
\end{displaymath}
\end{theo}
\begin{proof}
For large $n$,
\begin{displaymath}
\alpha_n^l=\frac{\left|\left(\underline{-q_n+1},\underline{1}\right)\right|}{\left|\left[\underline{-q_n+1},\underline{1}\right)\right|},
\end{displaymath}
which is certainly less than 
\begin{displaymath}
\textbf{Poin}\left(\underline{-q_n+1},(\underline{1},\underline{-q_{n-1}+1}]\right).
\end{displaymath}
By Theorem \ref{point} there exists a constant
\begin{displaymath}
K_{0,n}=e^{\sigma\left(\max_{j=0,\dots,q_{n-1}-2}|f^{j}((\underline{1},\underline{-q_{n-1}+1}])|\right)\sum_{j=0}^{q_{n-1}-2}|f^{j}(\underline{-q_n+1})|}
\end{displaymath}
such that
\begin{equation}\label{ato}
\alpha_n^l\leq K_{0,n}\delta_n(1)\cdot s_n(1),
\end{equation}
with
\begin{displaymath}
\delta_n(k)=\frac{\left|\left(\underline{-q_n+kq_{n-1}},\underline{kq_{n-1}}\right)\right|}{\left|\left[\underline{-q_n+kq_{n-1}},\underline{kq_{n-1}}\right)\right|}
\end{displaymath}
and
\begin{displaymath}
s_n(k)=\frac{\left|\left[\underline{-q_n+kq_{n-1}},\underline{0}\right]\right|}{\left|\left(\underline{-q_n+kq_{n-1}},\underline{0}\right]\right|}.
\end{displaymath}
By the Mean Value Theorem, $f$ transforms the intervals defining $\delta_n(k)$ into a pair whose ratio is
\begin{displaymath}
\left(\frac{u_k}{v_k}\right)\delta_n(k)
\end{displaymath}
where $u_k$ is the derivative of $x^l$ at a point in the interval $U_k$ between $\underline{-q_n+kq_{n-1}}$ and $\underline{kq_{n-1}}$,while $v_k$ is the derivative at a point in 
\begin{displaymath}
V_k=U_k\cup\underline{-q_n+kq_{n-1}}.
\end{displaymath}
Observe that for $n$ sufficiently large,
\begin{displaymath}
u_1<v_1<u_2<v_2<\dots<u_{a_n}<v_{a_n}.
\end{displaymath}
Moreover
\begin{displaymath}
f(\delta_n(k))\leq\textbf{Poin}\left(\underline{-q_n+kq_{n-1}+1},(\underline{kq_{n-1}+1},\underline{-q_{n-1}+1}]\right) 
\end{displaymath}
which by Theorem \ref{point} is less than
\begin{displaymath}
K_{k,n}\textbf{Poin}\left(\underline{-q_n+(k+1)q_{n-1}},(\underline{(k+1)q_{n-1}},\underline{0}]\right).
\end{displaymath}
So, the sequence of inequalities above can be rewritten in the form
\begin{equation}\label{sab}
\frac{u_k}{v_k}\delta_n(k)\leq K_{k,n}\cdot s_n(k+1)\cdot\delta_n(k+1).
\end{equation}
Now, starting from inequality (\ref{ato}) and applying $a_n-1$ times (\ref{sab}) we obtain
\begin{displaymath}
\left(\alpha_n\right)^l\leq K_{0,n}\dots\cdot K_{a_n-1,n}\delta_n(a_n)\cdot\frac{v_{a_n-1}}{u_1}\cdot s_n(1)\dots s_n(a_n).
\end{displaymath}
Note that $s_n(1)\dots s_n(a_n)\leq s_n$ and that
\begin{displaymath}
\frac{v_{a_n-1}}{u_1}\leq\left(\frac{\left|\left(\underline{-q_{n-2}},\underline{0}\right)\right|}{\left|\left(\underline{q_{n-1}},\underline{0}\right)\right|}\right)^{l-1}\leq\frac{\left|\left(\underline{-q_{n-2}},\underline{0}\right)\right|}{\left|\left(\underline{q_{n-1}},\underline{0}\right)\right|}
\end{displaymath}
hence
\begin{displaymath}
\left(\alpha_n\right)^l\leq K_{0,n}\dots\cdot K_{a_n-1,n}\delta_n(a_n)\cdot\frac{\left|\left(\underline{-q_{n-2}},\underline{0}\right)\right|}{\left|\left(\underline{q_{n-1}},\underline{0}\right)\right|}\cdot s_n
\end{displaymath}
which can be rewritten in the form 
\begin{equation}\label{4.}
\left(\alpha_n\right)^l\leq K_{0,n}\dots\cdot K_{a_n-1,n}s_n\nu_{n-2}\mu_{n-2}\alpha_{n-2},
\end{equation}
where
\begin{displaymath}
\nu_{n-2}=\frac{\left|\left[\underline{-q_{n-2}},\underline{0}\right)\right|}{\left|\left(\underline{q_{n-1}},\underline{0}\right)\right|}\cdot\frac{\left|\left[\underline{-q_{n-2}},\underline{0}\right)\right|}{\left|\left[\underline{-q_{n-2}},\underline{a_nq_{n-1}}\right)\right|},
\end{displaymath}
and
\begin{displaymath}
\mu_{n-2}=\frac{\left|\left(\underline{-q_{n-2}},\underline{a_nq_{n-1}}\right)\right|}{\left|\left(\underline{-q_{n-2}},\underline{0}\right)\right|}.
\end{displaymath}
To conclude the proof we have to estimate $\nu_{n-2}$ and $\mu_{n-2}$.\\
For $\nu_{n-2}$, observe that 
\begin{displaymath}
\left|\left[\underline{-q_{n-2}},\underline{0}\right)\right|\leq\left|\left(\underline{q_{n-3}},\underline{0}\right)\right|,
\end{displaymath}
so that
\begin{equation}\label{3.}
\nu_{n-2}\leq\frac{1}{\sigma_{n-1}\sigma_{n-2}}\cdot\frac{1}{1-\alpha_{n-2}}.
\end{equation}
To estimate $\mu_{n-2}$, we use an elementary lemma (Lemma 3.1 in \cite{a}).
\begin{lemma}
For any pair of numbers $x>y$ the inequality 
\begin{displaymath}
\frac{x^l-y^l}{x^l}\geq\left(\frac{x-y}{x}\right)\left(l-\frac{l(l-1)}{2}\left(\frac{x-y}{x}\right)\right)
\end{displaymath}
holds.
\end{lemma}
Now, apply $f$ to the intervals defining the ratio $\mu_{n-2}$. By the previous lemma, the resulting ratio is larger than 
\begin{displaymath}
\mu_{n-2}\left(l-\frac{l(l-1)}{2}\mu_{n-2}\right).
\end{displaymath}
The cross-ratio
\begin{displaymath}
\textbf{Poin}\left(\underline{-q_{n-2}+1},(\underline{a_nq_{n-1}+1},\underline{1})\right) 
\end{displaymath}
is larger again; so, by Theorem \ref{point}, there exists a constant $C_n$ such that
\begin{equation}\label{2.}
\mu_{n-2}\left(l-\frac{l(l-1)}{2}\mu_{n-2}\right)\leq C_n\cdot s_{n-1}\cdot\sigma_n\sigma_{n-1}.
\end{equation}
Solving this quadratic inequality we have that 
\begin{equation}\label{1.}
\mu_{n-2}\leq\frac{2}{l}\cdot\left(\frac{1}{1+\sqrt{1-\frac{2(l-1)}{l}C_ns_{n-1}\sigma_n\sigma_{n-1}}}\right)\cdot C_n\cdot s_{n-1}\cdot\frac{\sigma_n}{\sigma_{n-2}}.
\end{equation}
Combining (\ref{4.}), (\ref{3.}), (\ref{2.}) and (\ref{1.}) we conclude the proof.
\end{proof}

\subsection{Convergence of the Sequence $\alpha_n$}
The convergence of $\prod_{j=0}^{n}K_{i,j}$ and $\prod_{j=0}^{n}C_j$ is assured by the following lemma:  \\
\begin{lemma}\label{esp}
There exists $0<\lambda<1$ such that for every $i\in\{0,\dots,a_n-1\}$ and for every $n$ large enough, $K_{i,n}\leq e^{\sigma\left(\lambda^{n-2}\right)\lambda^{n-2}}$.
\end{lemma}
\begin{proof}
Observe that, for every $j\in\{0,\dots,q_{n-1}-2\}$
\begin{enumerate}
\item the intervals $f^{j}(\underline{-q_n+iq_{n-1}+1})$ are disjoint and $\sum_{j=0}^{q_{n-1}-2}\left|f^{j}(\underline{-q_n+iq_{n-1}+1})\right|$ is contained in a gap of the partition $\mathscr{P}_{n-1}$,
\item every interval $f^{j}((\underline{iq_{n-1}+1},\underline{-q_{n-1}+1}])$ is contained in a gap of the partition $\mathscr{P}_{n-2}$.
\end{enumerate}
So, by Corollary \ref{gap}, for large $n$
\begin{displaymath}
\sum_{j=0}^{q_{n-1}-2}\left|f^{j}(\underline{-q_n+iq_{n-1}+1})\right|<\lambda^{n-2}
\end{displaymath}
and
\begin{displaymath}
\sigma\left(\max_{j=0,\dots,q_{n-1}-2}|f^{j}\left(\left(\underline{iq_{n-1}+1},\underline{-q_{n-1}+1}\right]\right)|\right)<\sigma\left(\lambda^{n-2}\right).
\end{displaymath}
\end{proof}
A similar argument shows that also $C_n\leq e^{\sigma\left(\lambda^{n-2}\right)\lambda^{n-2}}$ for large $n$.\\

To establish the convergence of the sequence $\alpha_n$ we have to prove that $\prod_{i=0}^{n}\tilde{M_i}$ tends to zero.\\
Observe that $\tilde{M_n}\left(l\right)=M_n\left(l\right)T_n\left(l\right)$ with
\begin{displaymath}
M_n\left(l\right)=s_{n-1}^2\cdot\frac{2}{l}\cdot\left(\frac{1}{1+\sqrt{1-\frac{2(l-1)}{l}s_{n-1}\alpha_{n-1}}}\right)\cdot\frac{1}{1-\alpha_{n-2}}\cdot\frac{\sigma_n}{\sigma_{n-2}}
\end{displaymath}
and
\begin{displaymath}
T_n\left(l\right)=\frac{1+\sqrt{1-\frac{2(l-1)}{l}s_{n-1}\alpha_{n-1}}}{1+\sqrt{1-\frac{2(l-1)}{l}C_ns_{n-1}\alpha_{n-1}}}.
\end{displaymath}
The proof that $\prod_{j=0}^{n}M_j$ tends to zero is exactly the same as in \cite{a}. It remains to determine the size of $T_n(l)$.\\
Noting that
\begin{displaymath}
T_n(l)\leq\max{\left(1,\frac{\sqrt{1-\frac{2(l-1)}{l}s_{n-1}\alpha_{n-1}}}{\sqrt{1-\frac{2(l-1)}{l}C_ns_{n-1}\alpha_{n-1}}}\right)}
\end{displaymath}
it all comes down to the study of the function
\begin{displaymath}
f(x)=\frac{1-x}{1-C_nx}
\end{displaymath}
for $C_n> 1$ tending to one and $x=\frac{2(l-1)}{l}s_{n-1}\alpha_{n-1}$.\\
Observe that $f'(x)>0$ on $\left(0, c_n^{-1}\right)$ and that by the a priori estimates of \cite{a}, $x\leq 0.55$.\\
Therefore,
\begin{displaymath}
T_{n}(l)\leq 1+1.3\left(C_n-1\right).
\end{displaymath}
So, by Lemma \ref{esp}, $\prod_{i=0}^{n}T_i$ is bounded. \\

So, for $l\leq2$ the recursive inequality $\ref{alfa}$ implies that the sequence $\alpha_n$ goes to zero at least exponentially fast. Since $\alpha_n>\tau_n$ the same holds also for the scalings $\tau_n$.
\subsection*{Acknowledgments}
I would sincerely thank Prof. J. Graczyk for introducing me to the subject of this paper, his valuable advice and continuous encouragement. I am also very grateful to Prof. G. \'Swi\c atek for many helpful discussions.

\end{document}